\documentclass[12pt]{amsart}

\usepackage{bbm}

\usepackage{bbm}
\usepackage{enumerate}
\usepackage{amssymb}

\newtheorem{theorem}{Theorem}
\newtheorem{proposition}[theorem]{Proposition}
\newtheorem{lemma}[theorem]{Lemma}
\newtheorem{corollary}[theorem]{Corollary}

\theoremstyle{definition}

\theoremstyle{remark}

\newcommand \NN{\mathbbm{N}}
\newcommand \SI{{\rm Sym}(\NN)}
\newcommand \ACL{{\rm ACL}}
\newcommand \Aut{{\rm Aut}}
\newcommand \Sym{{\rm Sym}}
\newcommand \supp{{\rm supp}}

\begin{document}

\title{Rooted trees, strong cofinality and ample generics}
\author{Maciej Malicki}
\address{Institute of Mathematics, Polish Academy of Sciences, Sniadeckich 8, 00-956, Warsaw, Poland}
\email{mamalicki@gmail.com}
\date{June 07, 2010}

\begin{abstract}
We characterize those countable rooted trees whose full automorphism group has uncountable strong cofinality or contains an open subgroup with ample generics.

\end{abstract}

\maketitle
\section{Introduction}

  

In this paper, we study full automorphism groups of countable rooted trees, equipped with the standard product topology. We are mainly interested in the notions of strong cofinality and ample generics.

Recall that a group $G$ has uncountable strong cofinality if whenever $G$ is a union of a countable chain of subsets $A_0 \subset A_1 \subset \ldots$, then $A_m^k=G$ for some $k, m \in \NN$. This property was introduced by Bergman in \cite{Be}, and studied by authors such as Cournulier \cite{Co}, Droste and Holland \cite{DrHo} or Kechris and Rosendal \cite{KeRo}. It can also be viewed as a type of fixed point property, linking it to geometric group theory: it was observed in \cite{Co} that $G$ has uncountable strong cofinality if and only if every isometric action of $G$ on a metric space has bounded orbits. In particular, uncountable strong cofinality implies Serre's property (FA).

A separable and completely metrizable (that is, Polish) topological group $G$ has ample generics if the diagonal action of $G$ on $G^n$ by conjugation has a comeagre orbit for every $n \in \NN$. This notion was first studied by Hodges, Hodkinson, Lascar and Shelach in \cite{HoHo}, and later by Kechris in Rosendal in \cite{KeRo}. It is a very strong property: a group $G$ with ample generics, or even containing an open subgroup with ample generics, has the small index property, every homomorphism from $G$ into a separable group is continuous, every action of $G$ on a separable space is continuous, and there is only one Polish group topology on $G$ (see \cite{KeRo}.)

These two seemingly unrelated concepts have in fact something in common. For example, if $G$ has ample generics, and $G$ is a union of a countable chain of \emph{non-open subgroups} $G_0 \leq G_1 \leq \ldots$, then $G_m=G$ for some $m \in \NN$ (see Theorem 1.7, \cite{KeRo}.) However, in general none of them is implied by the other.

Our main results show that in the context of automorphism groups of rooted trees, both of these properties are strictly related to the behavior of the algebraic closures of finite sets, and that in a sense having ample generics is a strong form of having uncountable strong cofinality.

Let $\ACL_T(X)$ for $X \subseteq T$ denote the algebraic closure of $X$ in $T$, that is, the set of all elements of $T$ contained in a finite orbit under the action of the pointwise stabilizer $G_{\left\langle X\right\rangle}$ of $X$. Then we have

\begin{theorem}
\label{Th1}
Let $T$ be a countable rooted tree, $G=\Aut(T)$.
\begin{enumerate}
\item $G$ has uncountable strong cofinality iff $\ACL_T(\emptyset)$ is finite;
\item $G$ has an open subgroup with ample generics iff $\ACL_T(X)$ is finite for every finite $X \subseteq T$.
\end{enumerate}
\end{theorem}

It turns out that Theorem \ref{Th1} sheds light on the relationship between certain known results on rigidity of groups of automorphisms of trees, which we will discuss in the last section of the paper. In fact, this investigation was inspired by them.



\section{Notation and basic facts}
\paragraph{\textbf{Trees}}
By a \emph{rooted tree} $T$ with root $r$ we mean an ordering $(T,<)$ with the smallest element $r$, and such that all initial sets in $T$ are finite chains. Any such tree ordering $<$ is determined by the corresponding predecessor function $p$ (with $p(r)=r$.)

By a \emph{subtree} of $T$, we mean a subset $T' \subseteq T$ that is closed under the predecessor function. 

Every full automorphism group $\Aut(T)$ of a tree $T$ is assumed to be equipped with the pointwise convergence topology, which is easily seen to be separable and completely metrizable (that is, Polish.) If $G=\Aut(T)$ and $X \subset T$, then $G_{\left\langle X \right\rangle}$ stands for the pointwise stabilizer of $X$.

A \emph{leaf} in $X \subset T$ is an element with no successors in $X$. For a rooted tree $T$, and $t \in T$, $T_t$ is a rooted tree with root $t$, defined by $T_t= \{t' \in T: t' \geq t \}$.

In this paper (except for the last section), all trees are assumed to be rooted and countable.






\paragraph{\textbf{Wreath products}}
To avoid unnecessarily complicated notation, we will define the unrestricted generalized wreath product only for rooted trees. This definition agrees with the notion of generalized wreath product defined in \cite{Hol}, except that we reverse the underlying ordering. 

Let $T$ be a rooted tree, $N_t$, $t \in T$, be finite or countably infinite sets, and let $G_t$, $t \in T$, be permutation groups of $N_t$. We define $Wr_{t \in T} \, G_t$ as a group of permutations of $X=\prod_{t \in T} N_t$ satisfying the following conditions. 

For $x \in X$, $s,t \in T$, $i \in N_t$ define $x^t_i \in X$ by
\[ x^t_i(s)=\left\{ 
 \begin{array}{ cc }
     x(s) & {\rm if} \ s \neq t \\
     i & {\rm if} \ s=t
 \end{array}
 \right. 
.\]
Now, $g \in {\rm Sym}(X)$ is an element of $Wr_{t \in T} \, G_t$ if for every $x,y \in X$, and $t \in T$
\begin{enumerate}
\item $x(s)=y(s)$ for all $s<t$ $\Rightarrow g(x)(s)=g(y)(s)$ for all $s<t$;
\item there is $\sigma \in G_t$ such that $g(x^t_i)(t)=\sigma(i)$ for all $i \in N_t$.
\end{enumerate}



It is easy to see that $Wr_{t \in T} \, G_t$ is indeed a group.

For two groups $G_1$, $G_2$ of permutations of sets $N_1$, $N_2$, respectively, the standard wreath product $G=G_1 Wr \, G_2$ agrees with the above definition if we take $T=\{t_1,t_2 \}$, with $t_1>t_2$ and $G_{t_1}=G_1$, $G_{t_2}=G_2$. Then the \emph{base group} $G^{base}$ of $G$ is a normal group of all the permutations $g \in G$ such $g(n_1,n_2)=(n'_1,n_2)$ for every $(n_1,n_2) \in N_1 \times N_2$, that is, the second coordinate of $(n_1,n_2)$ stays fixed. It is easy to see that $G^{base} = G_1^{N_2}$, and that every element $g \in G$ is of the form $g=g_1g_2$, where $g_1 \in G^{base}$, $g_2 \in G_2$. In other words, $G=G^{base}G_2$.

Now, let $T$ be a rooted tree, $G=\Aut(T)$, and $B$ be the set of all cofinal branches in $T$. Then $G$ naturally permutes $B$. Let $B_0 \subseteq B$ be a set containing exactly one representative of every orbit of the action of $G$ on $B$, and let $\overline{T} \subseteq T$ be the tree induced by $B_0$. For $t \in \overline{T}$, put $N_t=\{0, \ldots, N-1\}$, where $N$ is the size of the orbit of $t$ under the action of $G_{\left\langle p(t)\right\rangle}$ on $T$, and $G_t=\Sym(N_t)$.

It is a straightforward exercise to show that $G \cong Wr_{t \in \overline{T}} \, G_t$. Let
\[Y=\{ (n_t)_{t \in I}: n_t \in N_t, \ I {\rm \ is \ an \ initial \ segment \ in \ } \overline{T} \}. \]

It is not hard to see that $Y$ with the ordering
\[ (n_t)_{t \in I} \leq_Y (m_t)_{t \in J} \Leftrightarrow I \subseteq J {\rm \ and \ } n_t=m_t {\rm \ for \ } t \in I \]
is isomorphic to $T$.

Now observe that $Wr_{t \in \overline{T}} \, G_t$ acts faithfully on $Y$ in such a manner that for $g \in \Sym(Y)$ we have that $g \in Wr_{t \in \overline{T}} \, G_t$ if and only if
\begin{enumerate}
\item $g((n_t)_{t \in I})=(m_t)_{t \in J}$ implies that $I=J$;
\item if $g((n_t)_{t \in I})=(m_t)_{t \in I}$ and $I' \subset I$ is an initial segment in $I$, then $g((n_t)_{t \in I'})=(m_t)_{t \in I'}$.
\end{enumerate}

Therefore, $G \cong Wr_{t \in \overline{T}} \, G_t$, and this is independent of the choice of $\overline{T}$. Nevertheless, for a tree $T$, by $\overline{T}$ we will always mean some fixed set as above.

Recall that a structure $M$ is \emph{ultrahomogeneous} if every isomorphism $\phi:A \rightarrow B$ between finite substructures of $M$ can be extended to an automorphism of $M$. 

Obviously, not all trees are ultrahomogeneous. We add to $T$ a family of unary predicates $P_t$, $t \in \overline{T}$, such that the structure $(T,p,\{P_t\}_{t \in \overline{T}})$ is ultrahomogeneous, locally finite, and has the same automorphisms as $T$:

\[ P_t(t') \leftrightarrow t'=g(t) {\rm \ for \ some \ } g \in \Aut(T) \]
for $t \in \overline{T}$, $t' \in T$.
  
\begin{proposition}
\label{prWr}
Let $T$ be a  tree, $\overline{T}$, $N_t,G_t$ for $t \in \overline{T}$ be defined as above.
The structure $(T,p,\{P_t\}_{t \in \overline{T}})$ is ultrahomogeneous, locally finite, and
\[ \Aut(T)=\Aut(T,p, \{P_t\}_{t \in \overline{T}}).\]
\end{proposition}

\begin{proof}
We show that if $f:A \rightarrow B$ is an isomorphism between finite subsets of $T$, and $t \in T \setminus A$ is an immediate successor of some $a \in A$, then there exists $t' \in T$ such that $f \cup \{ (t,t') \}$ is an isomorphism. By the standard back-and-forth argument this implies ultrahomogeneity of $(T,p,\{P_t\}_{t \in \overline{T}})$.

Let $f:A \rightarrow B$, $t, a \in T$ be as above, and fix $g \in \Aut(T)$ with $g(a)=f(a)$. Let $\mathcal{O}$ be the orbit of $t$ under the action of the stabilizer $G_{\left\langle a \right\rangle}$ on $T$. Then $t$ witnesses that $\mathcal{O} \setminus A \neq \emptyset$, so there exists $t' \in g[\mathcal{O}] \setminus f[A]$. Clearly, $t'$ is as required.

Local finiteness follows form the fact that initial sets in $T$ are finite chains. It is also easy to see that the predicates $P_t$ do not affect automorphisms of $T$.
\end{proof}

\paragraph{\textbf{Ample generics}.}
For a countable structure $M$ in a fixed countable signature $L$, the family $\mathcal{K}={\rm Age}(M)$ is the family of all finite substructures of $M$. Also, for $n \in \NN$, the family $\mathcal{K}_p^n$ consists of all objects of the form
\[ \left\langle A, \phi_1:B_1 \rightarrow C_1, \ldots, \phi_n:B_n \rightarrow C_n \right\rangle, \]
where $A \in \mathcal{K}$, $B_i,C_i \subseteq A$, and $\phi_i$ are isomorphisms $\phi_i:B_i \rightarrow C_i$, $i \leq n$.

There is a natural notion of embedding associated with every $\mathcal{K}_p^n$. For
\[ \mathcal{S}=\left\langle A, \phi_1:B_1 \rightarrow C_1, \ldots, \phi_n:B_n \rightarrow C_n \right\rangle, \]
\[ \mathcal{T}=\left\langle D, \psi_1:E_1 \rightarrow F_1, \ldots, \psi_n:D_n \rightarrow F_n \right\rangle, \]
$f:A \rightarrow D$ embeds $\mathcal{S}$ into $\mathcal{T}$ if it is an embedding of $A$ into $D$ as structures, and
\[ f \circ \phi_i \subseteq \psi_i \circ f \]
for $i \leq n$.
 
We say that $\mathcal{K}_p^n$ satisfies the \emph{weak amalgamation property} (WAP) if for every $\mathcal{S} \in \mathcal{K}_p^n$ there exists $\mathcal{T} \in \mathcal{K}_p^n$, and an embedding $e: \mathcal{S} \rightarrow \mathcal{T}$ such that for every $\mathcal{F}, \mathcal{G} \in \mathcal{K}_p^n$, and embeddings $i: \mathcal{T} \rightarrow \mathcal{F}$, $j:\mathcal{T} \rightarrow \mathcal{G}$, there exists $\mathcal{E} \in \mathcal{K}_p^n$ and embeddings $k:\mathcal{F} \rightarrow \mathcal{E}$, $l:\mathcal{G} \rightarrow \mathcal{E}$ such that $\mathcal{E}$ amalgamates $\mathcal{F}$ and $\mathcal{G}$ over $\mathcal{T}$, that is,
\[ k \circ i \circ e=l \circ j \circ e.\]

The family $\mathcal{K}_p^n$ satisfies the \emph{joint embedding property} (JEP) if any two $\mathcal{S},\mathcal{T} \in \mathcal{K}_p^n$ can be embedded in some $\mathcal{E} \in \mathcal{K}_p^n$.

A Polish group $G$ has \emph{ample generics} if each diagonal action of $G$ on $G^n$, $n \in \NN$, by conjugation:
\[ g.(g_0,\ldots, g_n)=(gg_0g^{-1}, \ldots, gg_ng^{-1}), \]
$g, g_1, \ldots, g_n \in G$, has a comeagre orbit. We have (Theorem 6.2 from \cite{KeRo}):
\begin{theorem}
\label{thAmple}
Let $M$ be a countable, locally finite, ultrahomogeneous structure, $\mathcal{K}={\rm Age}(M)$, and $G={\rm Aut}(M)$. Then $G$ has ample generics if and only if $\mathcal{K}_p^n$ satisfies WAP and JEP for every $n \in \NN$.
\end{theorem}
 
More information on countable structures, families of their finite substructures, and ultrahomogeneity (all closely related to the notion of the Fra\"{i}ss\'e limit) can be found in \cite{Ho}.

\paragraph{\textbf{Groups}} A group $G$ has \emph{uncountable strong cofinality} if for any $A_0 \subseteq A_1 \subseteq \ldots$ such that $G=\bigcup_m A_m$, we have $A_m^k=G$ for some $m$, $k$. If $G=\bigcup_m G_m$ for some strictly increasing sequence of subgroups, that is, $G_0<G_1<\ldots$, then we say that $G$ has \emph{countable cofinality}. Finally, a topological group $G$ has the \emph{small index property} if any subgroup of index less than $2^{\aleph_0}$ is open in $G$.

\section{Uncountable strong cofinality}

The first two lemmas are straightforward, so we omit their proofs.

\begin{lemma}
\label{le1}
Let $T$ be a tree, $\Aut(T)=Wr_{t \in \overline{T}} \, G_t$, and $\overline{S} \subseteq \overline{T}$ be a subtree of $\overline{T}$. Then $\overline{S}$ corresponds to an invariant subtree $S \subset T$, $\Aut(S)=Wr_{t \in \overline{S}} \, G_t$, and  $g \mapsto g_{\upharpoonright \overline{S}}$, $ g \in Wr_{t \in \overline{T}} \, G_t$, defines a continuous and surjective homomorphism $\phi:Wr_{t \in \overline{T}} \, G_t \rightarrow Wr_{t \in \overline{S}} \, G_t$.
\end{lemma}

\begin{lemma}
\label{le2}
Let $T=\{r\} \cup \{ t_i \}$ be a tree such that $r$ is the immediate predecessor of each $t_i$, and $G_t$, $t \in T$, be permutation groups. Then $Wr_{t \in T} \, G_t$ is isomorphic to $(\otimes_i G_{t_i}) Wr \, G_r$.
\end{lemma}

\begin{lemma}
\label{le3}
Let $T$ be a finite tree, and let $G_t$, $t \in T$, be permutation groups. If every $G_t$, $t \in T$, has uncountable strong cofinality (in, particular, if $G_t$ is finite), then $Wr_{t \in T} \, G_t$ has uncountable strong cofinality.
\end{lemma}

\begin{proof}
This is an easy induction on the size of $T$. Let $t_0 \in T$ be such that all successors of $t_0$, say $t_1, \ldots, t_n$, are leaves in $T$. Let $T'=T \setminus \{t_1, \ldots, t_n \}$, and let $G'_{t_0}=(\otimes_{i=1}^n G_{t_i}) Wr \, G_{t_0}$, $G'_t=G_t$ if $t \in T'$, $t \neq t_0$. Clearly, $G_{t_0}$ has uncountable strong cofinality, and $\left| T' \right|<\left| T \right|$. By induction hypothesis, $Wr_{t \in T'} \, G'_t$ has uncountable strong cofinality, and Lemma \ref{le2} implies that $Wr_{t \in T'} \, G'_t$ is isomorphic to $Wr_{t \in T} G_t$ 
\end{proof}

The next lemma contains folklore facts.

\begin{lemma}
\label{le3.5}
The group $(\mathbbm{Z}_2)^\NN$ has countable cofinality, and does not have the small index property. Therefore, every separable, and completely metrizable topological group $G$ that maps homomorphically, continuously and surjectively onto $(\mathbbm{Z}_2)^\NN$ has countable cofinality, and does not have the small index property.
\end{lemma}

\begin{proof}
Select a Hamel basis $B$ for $(\mathbbm{Z}_2)^\mathbbm{N}$ regarded as a linear space, and build a countable strictly increasing sequence $B_0 \subsetneq B_1 \subsetneq \ldots$ such that $B=\bigcup_n B_n$. The linear spaces $H_n$ generated by $B_n$ are groups witnessing countable cofinality of $(\mathbbm{Z_2})^\NN$.

Now, let $K$ be the complement of a non-principal ultrafilter on $\NN$. Then $K$ is a subgroup of $(\mathbbm{Z}_2)^\NN$ of index $2$, dense in $(\mathbbm{Z}_2)^\NN$, and so it cannot be open. Thus, $K$ witnesses that $(\mathbbm{Z})^\NN$ does not have the small index property.

If $G$ is separable and completely metrizable, and $\phi:G \rightarrow (\mathbbm{Z}_2)^\NN$ is homomorphic, continuous and surjective, then $\phi$ is open (see Theorem 2.3.3 in \cite{Gao}), so $\phi^{-1}[K]$ is not open in $G$, and has index $2$ in $G$. Also, groups $\phi^{-1}[H_n]$ form a strictly increasing sequence, whose union is $G$. 
\end{proof}

\begin{lemma}
\label{le4}
Let $S_0$ be a  tree defined by
\begin{enumerate}
\item  $S_0$ is an infinite branch, or
\item  $S_0=\{t_0, \ldots, t_n \} \cup \{s_0,s_1,\ldots \}$, where $r=t_0,t_1, \ldots, t_n$ is the unique path joining the root $r$ and the only non-trivially branching element $t_n$, and $t_n$ is the immediate predecessor of each $s_0, s_1, \ldots$.
\end{enumerate}
Suppose that $T$ is a tree, $\Aut(T)=Wr_{t \in \overline{T}} \, G_t$, and $\overline{T}$ contains a subtree $S_0$ as above such that each $G_t$, $t \in S_0$, is finite. Then $\Aut(T)$ has countable cofinality, and does not have the small index property.
\end{lemma}

\begin{proof}
Suppose that $\overline{T}=S_0$. Without loss of generality, we can assume that all $G_s$, $s \in S_0$, are non-trivial. Suppose that $S_0$ is an infinite branch. Let $S_0=\{s_0,s_1,\ldots \}$ be the increasing enumeration of $S_0$, and let $H_n=Wr_{s \in \{s_0, \ldots, s_n \}} \, G_s$. Each $H_n$ is a finite permutation group of a finite set, so each of its elements $g$ can be homomorphically assigned its sign, ${\rm sgn}(g)$. In this manner, we get a continuous and surjective homomorphism $\Aut(T) \rightarrow (H_n)^\NN \rightarrow (\mathbbm{Z}_2)^\NN$, so $Wr_{s \in S_0} \, G_t$ has countable cofinality and does not have the small index property by Lemma \ref{le3.5}.

Let us consider the other case now.
Let $H=G_{t_n} Wr \ldots Wr \, G_{t_0}$. By Lemma \ref{le2}, the group $G$ can be written as $G=(\otimes_n G_{s_n}) Wr \, H$, where $H$ is a group of permutations of a finite set of size $N+1$, and $G_{s_n}$ are symmetric groups. Hence, every element $g \in G^{base}$ is of the form
\[ ((g^0_0,g^0_1,\ldots),(g^1_0,g^1_1,\ldots),\ldots ,(g^N_0,g^N_1 \ldots)), \]
where $g^i_j \in G_{s_i}$. Therefore, we can define $\phi: G \rightarrow (\mathbbm{Z}_2)^\NN$ by 
\[ \phi(hg)=({\rm sgn}(g^0_n\ldots g^N_n))_{n \in \NN} \]
for $h \in H$, $g \in G^{base}$. The mapping $\phi$ is a homomorphism. To see this, note that for every $h \in H$, $g \in G^{base}$, $gh=h\bar{g}$, where $\bar{g}$ is a coordinate permutation of $g$ of the form
\[ ((g^{i_0}_0,g^{i_0}_1,\ldots),(g^{i_1}_0,g^{i_1}_1,\ldots),\ldots ,(g^{i_N}_0,g^{i_N}_1 \ldots)), \]
so
\[ \phi(h_0g_0h_1g_1)=\phi(h_0h_1\bar{g}_0g_1)=\phi(\bar{g}_0g_1)=\phi(g_0g_1)=\phi(h_0g_0)\phi(h_1g_1) \] 
for every $h_0,h_1 \in H$, $g_0,g_1 \in G^{base}$.

As all $G_s$, $s \in S_0$, are symmetric groups, it is also surjective, so, as before, $G$ can be homomorphically, continuously, and surjectively mapped onto $(\mathbbm{Z}_2)^\NN$.

If $S_0$ is a subtree of $\overline{T}$, then by Lemma \ref{le1}, $Wr_{t \in \overline{T}} \, G_t$ maps homomorphically, continuously and surjectively onto $Wr_{t \in S_0} \, G_t$. An application of Lemma \ref{le3.5} finishes the proof. 
\end{proof}

Now we prove the main technical lemma.

\begin{lemma}
\label{leMain}
Suppose that $G=\otimes_{n=0}^N (G_n)^\NN$, where $G_n$ are any groups, $N \in \{0,1, \ldots, \NN\}$, and $G=\bigcup_m A_m$, where $A_0 \subseteq A_1 \subseteq \ldots$ are subsets of $G$. Suppose also that the following condition is satisfied:

there exists $l$ such that if $g=(g_0,g_1\ldots) \in A^k_m$ for some $m$, $k$, and $\bar{g}=(\bar{g}_0,\bar{g}_1\ldots) \in G$ is such that each $\bar{g}_n \in G_n^{\NN}$ is a coordinate permutation of $g_n$, then $\bar{g} \in A_m^{k+l}$.

Then $G=A_m^k$ for some $m, k$.
\end{lemma}

\begin{proof}
Without loss of generality we can assume that groups $G_n$ are pairwise disjoint, and that $e \in A_0$. We start with a claim.

\emph{Claim}. There exist $k$ and $g^C \in G$, $C$ is a countable subset $C$ of $\bigcup_n G_n$ (that is, $C \in [\bigcup_n G_n]^{\omega}$), such that if $g^C \in A_m$ and $g \in G$ with range$(g) \subseteq C$ then $g \in A^k_m$.

In the proof of the claim, by saying that $g_{\upharpoonright I}$ is a coordinate permutation of $h$, where $g, h \in G$, $I \subseteq \NN^N$, we mean, abusing terminology slightly, that there exists a bijection $f :I \rightarrow \NN^N$ such that $(g_{f(i)})=h$.

For $C$ as above, take $g^C=(g^C_0,g^C_1, \ldots)$ to be some fixed element such that each $g^C_n$ contains infinitely many copies of every element from $(C \cap G_n) \cup (C \cap G_n)^{-1} \cup \{ e \}$.

Suppose first that $g=(g_0,g_1,\ldots) \in G$ is such that each $g_n$ has infinitely many coordinates equal to the identity.

We can partition $\NN^N$ into $3$ infinite subsets $I,I',I''$ such that $g^C_{\upharpoonright I}$ is a coordinate permutation of $g$, $g^C_{\upharpoonright I'}$ is a coordinate permutation of $g^C$, and $g^C_{\upharpoonright I''}$ is the identity. Then, by the definition of $g^C$, the element $g^C_{\upharpoonright I \cup I'}$ is a coordinate permutation of $(g^C)^{-1}$, and $g^C_{\upharpoonright I' \cup I''}$ is a coordinate permutation of $g^C$. Fix a permutation $\sigma$ of $\NN^N$ such that $\sigma[I'']=I$, $\sigma[I \cup I' ]=I' \cup I''$, and $(g^C_\sigma)_{\upharpoonright I' \cup I''}=(g^C)^{-1}_{\upharpoonright I' \cup I''}$, where $g^C_\sigma$ is a coordinate permutation of $g^C$ induced by $\sigma$. Then, by our assumption, $g^C_\sigma \in A_m^{1+l'}$ for some $l'$, $g^C g^C_\sigma$ is a coordinate permutation of $g$, and $g^C g^C_\sigma \in A^{2+l'}_m$, so, by our assumption again, $g \in A^{2+l'+l''}_m$ for some $l''$. Here $l', l''$ are independent of the choice of $C$ and $g$.

Since any $g \in G$ can be expressed as a product of two elements as above, $g \in A_m^k$, if $k \geq 2(2+l'+l'')$ and range$(g) \subset C$.



Put
\[ B_m= \{ C \in [\bigcup_n G_n]^\omega: g^C \in A_m \}. \]

Clearly, $\bigcup_m B_m=[\bigcup_n G_n]^\omega$. But this means that there exists $m$ such that $B_m=B$. Otherwise, there is some $C_m \notin B_m$ for every $m$. Since families $B_m$ are closed under taking subsets, we have that $\bigcup_m C_m \notin B_n$ for every $n$, and $\bigcup_m B_m \neq B$, which is a contradiction. By the claim, $G=A^k_m$ for some $k$.
\end{proof}


\begin{lemma}
\label{leMain2}
Let $N \in \{0,1, \ldots, \NN \}$, and let $G_n$, $n \in N$, be permutation groups. Then $G=\otimes_{n=0}^N (G_n Wr \, \SI)$ has uncountable strong cofinality.
\end{lemma}

\begin{proof}


Every element of $G=\otimes_{n=0}^N (G_n Wr \, \SI)$ is of the form
\[ (g_0s_0,g_1s_1,\ldots),\]
where $g_n \in G_n^\NN$, $s_n \in \SI$ , $n \in N$, so we can write it as
\[ (g_0,g_1, \ldots)(s_0,s_1, \ldots),\]
where $(g_0,g_1,\ldots) \in \otimes_{n=0}^N G_n^\NN$, $(s_0,s_1,\ldots) \in (\SI)^N$.

Suppose that $G=\bigcup_m  A_m$, where $A_0 \subseteq  A_1 \subseteq \ldots $, and put $H=\otimes_{n=0}^N G_n^\NN$, $B_m=H \cap A_m$. By Lemma 3.5 from \cite{DrGo}, the group $(\SI)^N$ has uncountable strong cofinality, that is, there exist $m,k$ such that $(\SI)^N \subseteq A_m^k$. Obviously, without loss of generality we can assume that $m=0$. Observe that the natural action of $(\SI)^N$ on $H$ by conjugation gives rise to all possible permutations of coordinates of elements of $H$. Therefore, Lemma \ref{leMain} implies that there exist $m,k$ such that $H \subseteq B_m^k$, and $G$ has uncountable strong cofinality.
\end{proof}





\begin{theorem}
\label{ThM1}
Let $T$ be a  tree. If $\ACL_T(\emptyset)$ is finite, then $\Aut(T)$ has uncountable strong cofinality. Otherwise, it has countable cofinality and does not have the small index property.
\end{theorem}

\begin{proof}
Put $G=\Aut(T)=Wr_{t \in \overline{T}} \, G_t$, and define
\[ S=\overline{T} \setminus \bigcup \{\overline{T}_t: t \in \overline{T} {\rm \ and \ } G_t \ {\rm is \ infinite} \}. \]

It is straightforward to check that $S$ is finite if and only if $\ACL_T(\emptyset)$ is finite.

Suppose that $S$ is finite, and fix $s \in S$. Let $\{t_n\}$ be an enumeration of all immediate successors of $s$ in $\overline{T}$ such that $t_n \in \overline{T} \setminus S$, that is, $G_{t_n}=\SI$. Then, for every tree $\overline{T}_{t_n}$ with root $t_n$, we have
\[ \Aut(T_{t_n})=Wr_{t \in \overline{T}_{t_n}} \, G_t= H_n \, Wr \, \SI, \]
for some permutation group $H_n$, so, by Lemma \ref{le2},
\[ \Aut(\{s\} \cup \{\overline{T}_{t_n}\})=Wr_{t \in \{s\}\cup\{\overline{T}_{t_n}\}} \, G_t=(\otimes_n (H_n \, Wr \, \SI)) \, Wr \, G_s. \]

We add a new element $s'$ to $S$, which is an immediate successor of $s$, and put $G_{s'}=\otimes_n (H_n \, Wr \, \SI)$.

Then, for $S'=S \cup \{s': s \in S\}$, the group $Wr_{s \in S'} \, G_s$ is isomorphic to $G$. By Lemma \ref{leMain2}, each $G_{s'}$ has uncountable strong cofinality. Since each $G_s$, $s \in S$, is finite, and $S'$ is finite, by Lemma \ref{le3}, $Wr_{s \in S'} \, G_s$, and so $G$, has uncountable strong cofinality.



If $S$ is infinite, then, by K\"{o}nig's lemma, $S$ contains an infinite branch, or some element of $S$ has infinitely many immediate successors; in any case, $S$ and thus $\overline{T}$ contains a subtree $S_0$ as in the statement of Lemma \ref{le4}. By Lemma \ref{le4}, $G$ has countable cofinality and does not have the small index property.

\end{proof}

\section{Ample generics.}


\begin{theorem}
\label{ThM2}
Let $T$ be a  tree.
For $G={\rm Aut}(T)$ the following conditions are equivalent:
\begin{enumerate}
\item $\ACL_T(X)$ is finite for every finite $X \subset T$;
\item $G$ contains an open subgroup $H$ with ample generics;
\item $G$ has the small index property.
\end{enumerate}
\end{theorem}

\begin{proof}
We show $(1) \Rightarrow (2)$. Suppose that $\ACL_T(X)$ is finite for every finite $X \subseteq T$, and let $X_0=\ACL_T(\emptyset)$. Observe that the stabilizer $G_{\left\langle X_0\right\rangle}$ of $X_0$ is open in $G$, and $G_{\left\langle X_0 \right\rangle}$ can be thought of as the automorphism group of an ultrahomogeneous, locally finite structure $T'$ obtained from $T$ by adding names to $T$ for every $x \in X_0$.

We show that for $\mathcal{K}={\rm Age}(T')$, the classes $\mathcal{K}_p^n$, $n \in \NN$, satisfy WAP and JEP. 

Fix $n \in \NN$, and for $\mathcal{S} \in \mathcal{K}_p^n$ of the form

\[ \mathcal{S}=\left\langle A, \phi_1:B_1 \rightarrow C_1, \ldots, \phi_n:B_n \rightarrow C_n \right\rangle \]
let $A'=\ACL_T(A)$, and let $\mathcal{T} \in \mathcal{K}_p^n$ be defined by

\[ \mathcal{T}=\left\langle A', \phi_1:B_1 \rightarrow C_1, \ldots, \phi_n:B_n \rightarrow C_n \right\rangle. \]

Then, by our assumption, for every successor $t$ of a leaf $a$ in $A$ with $P_{t'}(t)$ for some $t' \in \overline{T}$, there are infinitely many $t_n$, $n \in \NN$ such that $P_{t'}(t_n)$. Therefore, if
\[ \mathcal{F}=\left\langle H, \chi_1:M_1 \rightarrow N_1, \ldots, \chi_n:M_n \rightarrow N_n \right\rangle,\]
\[ \mathcal{G}= \left\langle P, \xi_1:Q_1 \rightarrow R_1, \ldots, \xi_n:Q_n \rightarrow R_n \right\rangle, \]
$\mathcal{F},\mathcal{G} \in \mathcal{K}_p^n$,
and $i:\mathcal{T} \rightarrow \mathcal{F}$, $j:\mathcal{T} \rightarrow \mathcal{G}$ are embeddings, we can assume without loss of generality that $H \cap P=A'$, and $(\chi_i)_{\upharpoonright A'}=(\xi_i)_{\upharpoonright A'}$, $i \leq n$. It is a little tedious but completely straightforward to check that in this case the structure $\mathcal{E} \in \mathcal{K}_n^p$ defined by
\[ \mathcal{E}=\left\langle H \cup P, \chi_1 \cup \xi_1, \ldots, \chi_n \cup \xi_n \right\rangle \]
along with natural embeddings $k$, $l$ of $\mathcal{F}$, $\mathcal{G}$ into $\mathcal{E}$ amalgamates $\mathcal{F}$ and $\mathcal{G}$ over $\mathcal{T}$.

To show JEP, observe that every embedding $f:A \rightarrow B$ between finite subsets of $T'$ fixes all elements in $X_0$, so we can repeat the above argument.

The implication $(2) \Rightarrow (3)$ follows from Theorem 6.9 from \cite{KeRo} saying that if $H \leq G$ has ample generics, then $H$ has the small index property, and an observation that, since $[H:G] \leq \aleph_0$, in this case $G$ also has the small index property.

Finally, we show $\neg (1) \Rightarrow \neg (3)$.
Let $X_0 \subset T$ be a finite set such that $\ACL_T(X_0)$ is infinite. It is not hard to find a tree $T'$, and symmetric groups $G_t$, $t \in T'$, such that $G_{\left\langle X_0 \right\rangle}=Wr_{t \in T'} \, G_t$. Then we can define $S \subset T'$ for $T'$ as in the proof of Theorem \ref{ThM1}, and observe that because $\ACL_T(X_0)$ is infinite, $S$ is also infinite. Therefore, there exists $S_0 \subset S$ as in Lemma \ref{le4} such that $G_s$ is a finite symmetric group for every $s \in S_0$. Lemma \ref{le4} implies that $G_{\left\langle X_0 \right\rangle}$, and so $G$, does not have the small index property.

\end{proof}
 
By Theorems 6.24 and 6.25 from \cite{KeRo}, we get
\begin{corollary}
Let $T$ be a  tree, $G=\Aut(T)$. If $\ACL_T (X)$ is finite for every finite $X \subset T$, then
\begin{enumerate}
\item every homomorphism from $G$ into a separable topological group $H$ is continuous;
\item the standard product topology is the unique Polish topology on $G$.
\end{enumerate}
\end{corollary}


\section{Rigidity of trees}

By a (non-rooted) tree, we mean a connected graph with no cycles.

In Theorem 4.4 from \cite{BaLu}, Bass and Lubotzky proved a rigidity theorem to the extent that a reach enough group of automorphisms of a locally finite tree $T$ completely determines $T$. That is, if $G$ is a group of automorphisms of locally finite trees $T_1,T_2$, satisfying some additional assumptions, we will not dwell into, then $T_1$ is isomorphic to $T_2$. 


As the authors pointed out, the condition of being locally finite is rather restrictive. This limitation was removed, applying two different approaches, by Psaltis (\cite{Ps}) and Forester (\cite{Fo}), however not without some trade-ins. Psaltis managed to get rid of the assumption of local finitiness of $T$, but had to restrict himself to full automorphism groups. His main result (Theorem 6.9a,b,c) is

\begin{theorem}[Psaltis]
Let $T$ be a tree with countable number of edges incident at each vertex, and $i_G(e) \geq 3$ for each edge $e$. Then ${\rm Aut}(T)$ completely determines $T$.
\end{theorem}

Here $i_G(e)$ denotes $[G_t:G_e]$, where $G_e$ is the stabilizer of edge $e$ in $G$, and $G_t$ is the stabilizer of vertex $t$ such that $e=(t,s)$ for some $s \in T$.

On the other hand, Forester's results concern also subgroups of the full automorphism groups, but with more additional assumptions present. In particular, they involve Serre's property (FA). Recall that $G$ has property (FA) if every action of $G$ on a tree without inversions has a fixed point.

\begin{theorem}[Forester]
Let $G$ be a group acting on trees $T_1, T_2$ without inversions. Let $T_1$ be a strongly slide-free, and
$T_2$ a proper tree, both cocompact. Suppose that all vertex stabilizers are
unsplittable. If either
\begin{enumerate}[a)]
\item one of the trees has (FA) vertex stabilizers, or
\item one of the trees is locally finite,
\end{enumerate}
then there is a unique isomorphism of $G$-trees $T_1$ and $T_2$.
\end{theorem}

We will not define all the technical notions involved in the statement of this theorem. Suffices to say that if $i_G(e) \geq 3$ for every edge $e$ in $T$, then $T$ is strongly slide-free and proper, and the stabilizers of vertices in ${\rm Aut}(T)$ are known to be unsplittable. Cocompactness means that there are only finitely many orbits of the action of the stabilizer of $t$ on the set of all children of $t$, $t \in T$, so it puts extra restrictions on $T_1, T_2$, compared to Psaltis' theorem. However, one can ask whether the assumption on sharing property (FA) by all stabilizers can be removed if $G$ is the full automorphism group (or, when it is satisfied.) Because uncountable strong cofinality clearly implies property (FA), Theorem \ref{Th1} shows that this happens only in very special situations, when  $i(e)=\aleph_0$ for `most' edges $e$ in $T$. It turns out that in this case stabilizers of vertices of $T$ contain an open subgroup with ample generics.

\begin{theorem}
\label{coNonRooted}
Let $T$ be a non-rooted tree. The following conditions are equivalent:
\begin{enumerate}

\item The stabilizer in ${\rm Aut}(T)$ of every vertex $t \in T$ has property (FA);
\item $\ACL_T(X)$ is finite for every finite non-empty $X \subset T$;
\item the stabilizer of every vertex $t \in T$ contains an open subgroup with ample generics.
\end{enumerate}
\end{theorem}

\begin{lemma}
\label{lePo}
Let $T$ be a non-rooted tree.  if $\ACL_T(\{t \})$ is finite for every $t \in T$ then $\ACL_T(X)$ is finite for every non-empty finite $X \subset T$.
\end{lemma}

\begin{proof}
Let $G=\Aut(T)$, $X \subset T$ be finite, non-empty, and for $x \in X$, let
\[ S_x=\{ t \in T: [x,t] \cap X=\{x\} \}. \]
where $[x,t]$ is the unique path of vertices in $T$ joining $x$ and $t$. Then each $S_x$ is a tree, and 
\[ g \in G_{\left\langle X \right\rangle} \Leftrightarrow \, \forall x \in X g[S_x]=S_x. \]

Fix $x \in X$. We show that if $t \in S_x$ is in an infinite orbit of $G_{\left\langle x \right\rangle}$, then it is in an infinite orbit of $G_{\left\langle X \right\rangle}$. This will finish the proof.

Let $g_n \in G_{\left\langle x \right\rangle}$, $n \in \NN$, witness that the orbit of $t$ under the action of $G_{\left\langle x \right\rangle}$ is infinite, $t_n=g_n(t)$, and, for the unique neighbor $s$ of $x$ with $s \in [x,t]$, let $s_n=g_n(s)$, $n \in \NN$. If $\{s_n\}$ is infinite, then by finiteness of $X$, $s_n \in S_x$, and so $t_n \in S_x$, for almost all $n$. If $\{ s_n \}$ is finite, then there are exists $n_0$ such that $t_{n_0} \in S_x$ and $t_{n_0} \in [x,s_n]$ for infinitely many $n$, that is, $s_n \in S_x$ for infinitely many $n$.

Now, it suffices to observe that if $g(t)=t'$ for some $t, t' \in S_x$ and $g \in G_{\left\langle x \right\rangle}$, then there exists $h \in G$ such that $h(t)=t'$ and $\supp(h) \subset S_x$, that is, $h \in G_{\left\langle X \right\rangle}$.
\end{proof}
   
\begin{proof}
In view of Theorems \ref{ThM1} and \ref{ThM2}, implications $(2) \Rightarrow (3)$ and $(3) \Rightarrow (1)$ are obvious. We show $(1) \Rightarrow (2)$. Since every $G_{\left\langle t\right\rangle}$ has property (FA), by Theorem 15 from \cite{Se} and remarks following it, no $G_{\left\langle t\right\rangle}$ has countable cofinality, that is, by Theorem \ref{ThM1}, $\ACL_T(\{t\})$ is finite for every $t \in T$. By Lemma \ref{lePo}, $\ACL_T(X)$ is finite for every finite non-empty $X \subset T$.
\end{proof}

\end{document}